\documentclass[a4paper,12pt,reqno]{amsart}
\usepackage[body={17cm,21cm}]{geometry}
\usepackage[initials]{amsrefs}

\usepackage{amssymb,amscd}
\usepackage[all]{xy}

\usepackage[mathscr]{eucal}
\usepackage{enumerate}

\numberwithin{equation}{section}
\theoremstyle{plain}
\newtheorem{thm}[equation]{Theorem}

\newtheorem{lem}[equation]{Lemma}
\newtheorem{prop}[equation]{Proposition}

\newtheorem{exa}[equation]{Example}
\newtheorem{rem}[equation]{Remark}

\theoremstyle{definition}

\theoremstyle{remark}

\providecommand{\R}[1]{\mathrm{#1}}

\DeclareMathOperator{\Gal}{Gal}

\DeclareMathOperator{\SL}{SL}
\DeclareMathOperator{\SO}{SO}

\DeclareMathOperator{\Spec}{Spec}
\DeclareMathOperator{\Spin}{Spin}

\def\R{{\mathbb R}}

\def\Gal{{\rm Gal}}

\DeclareFontEncoding{OT2}{}{} 

\usepackage{tikz-cd} 
\usepackage[colorlinks,linkcolor=black, anchorcolor=black, citecolor=black]{hyperref} 

\newcommand{\ep}{\varepsilon}

\newcommand{\rmH}{\mathrm{H}}
\newcommand{\rmK}{\mathrm{K}}
\newcommand{\rmU}{\mathrm{U}}

\newcommand{\bbQ}{\mathbb{Q}}
\newcommand{\bbR}{\mathbb{R}}
\newcommand{\bbC}{\mathbb{C}}

\DeclareMathOperator{\an}{an}
\DeclareMathOperator{\diag}{diag}
\DeclareMathOperator{\Nm}{Nm}
\DeclareMathOperator{\spl}{spl}
\DeclareMathOperator{\st}{st}
\DeclareMathOperator{\SU}{SU}

\newcommand{\norm}[1]{\left\lVert#1\right\rVert}

\DeclareFontFamily{U}{wncy}{}
\DeclareFontShape{U}{wncy}{m}{n}{%
<5>wncyr5%
<6>wncyr6%
<7>wncyr7%
<8>wncyr8%
<9>wncyr9%
<10>wncyr10%
<11>wncyr10%
<12>wncyr6%
<14>wncyr7%
<17>wncyr8%
<20>wncyr10%
<25>wncyr10}{}
\DeclareMathAlphabet{\cyr}{U}{wncy}{m}{n}
\begin{document}

\title[Counting integral points on indefinite ternary quadratic equations]
{Counting integral points on indefinite ternary quadratic equations over number fields}

\author{Fei Xu}
\address{Fei Xu \newline School of Mathematical Sciences, \newline Capital Normal University,
\newline 105 Xisanhuanbeilu, \newline 100048 Beijing, China}

\email{xufei@math.ac.cn}

\author{Runlin Zhang}

\address{Runlin Zhang \newline Beijing International Center for Mathematical Research, \newline Peking University,
\newline 1000871 Beijing, China}

\email{zhangrunlinmath@outlook.com}

\date{\today.}



\maketitle

\section*{\it Abstract}

We study an asymptotic formula for counting integral points over an equation defined by a non-degenerated indefinite integral ternary quadratic form $f$ representing a non-zero integer $a$ such that $-a\cdot det(f)$ is square over a number field. In particular, we prove that the finite part of this asymptotic formula is given by the product of local density times $1-p^{-1}$ over all finite primes $p$ over $\Bbb Z$.

\section{Introduction}

Let $f(x_1, \cdots, x_n)$ be a non-singular indefinite quadratic form over $\Bbb Z$ and $a$ be a non-zero integer. It is a classical question to count the integral solutions 
$$ N(f, a, T)= \# \{(\alpha_1, \cdots, \alpha_n)\in \Bbb Z^n: \ f(\alpha_1, \cdots, \alpha_n)=a  \ \text{with} \ \norm{(\alpha_1, \cdots, \alpha_n)} \leq T \} $$ as $T\to \infty$, where $\norm{ \cdot } $ is a norm on $\Bbb R^n$. 
For example, when $n=2$, it is well-known that 
$$ N(f, a, T) \sim \begin{cases}  O(1) \ \ \ & \text{when $-\det(f) \in (\Bbb Q^\times)^2$} \\
c \cdot log T & \text{otherwise} \end{cases}  $$
as $T\to \infty$ by the Dirichlet's  unit theorem (see \cite[Chapter 5, Theorem 5.12]{PR94}). In \cite{DRS}, Duke, Rudnick and Sarnak proved that 
$$ N(f, a, T) \sim c \cdot T^{n-2} $$ as $T \to \infty$ when $n\geq 4$. Eskin and McMullen provided a much simpler proof in \cite{EM}.  This result has been largely generalized by Eskin, Mozes and Shah  in \cite{EMS}. For $n=3$, Duke, Rudnick and Sarnak in \cite{DRS} also pointed out that 
$$ N(f, a, T) \sim \begin{cases}  c \cdot T \ \ \ & \text{$-a\det(f) \not\in (\Bbb Q^\times)^2$} \\
c \cdot T log T \ \ \ & \text{$-a\det(f) \in (\Bbb Q^\times)^2$} \end{cases}  $$
as $T \to \infty$ for some specific quadratic forms $f$. Such general phenomenon was proved by Oh and Shah in \cite{OS}.

In spirit of the local-global principle, one expects that the above constant $c$ should be the contribution of numbers of local solutions over all finite primes. Indeed, Borovoi and Rudnick in \cite{BR95} first proved that $c$ is the product of numbers of local solutions over all finite primes for $n\geq 4$ and studied the case of $n=3$ and $-a\det(f) \not\in (\Bbb Q^\times)^2$ by introducing the density function.  By applying strong approximation with Brauer-Manin obstruction, Wei and the first author showed that $c$ is an average of products of number of twisted local solutions by the Brauer elements over all finite primes in \cite{WX} (see also \cite{CX}) for the later situation. 
Since Oh and Shah in \cite{OS} provided a different type asymptotic formula for $n=3$ and $-a\det(f) \in (\Bbb Q^\times)^2$, it is natural to ask how the numbers of local solutions over finite primes contribute to the constant $c$ in the asymptotic formula. In this paper, we will answer this question (see Theorem \ref{main}). 

\begin{thm} \label{intro-m} When $n=3$ and $-a\det(f) \in (\Bbb Q^\times)^2$, then 
$$ N(f, a, T) \sim (\prod_{ p \ \text{primes}}(1-p^{-1}) \alpha_p(f,a) )\cdot  (log T \int_{B_T} \omega) $$ as $T\to \infty$ where 
$$\alpha_p(f,a)= \lim_{k\to \infty} \frac{\#\{ (\beta_1, \beta_2, \beta_3)\in (\Bbb Z/(p^k))^3: \ f(\beta_1, \beta_2, \beta_3) \equiv a \mod p^k\}}{p^{2k}}$$ for all primes $p$ and 
$$ B_T= \{ (\alpha_1, \alpha_2, \alpha_3) \in \Bbb R^3: \ f(\alpha_1, \alpha_2, \alpha_3)= a \ \text{with} \ \norm{ (\alpha_1, \alpha_2, \alpha_3)} \leq T \} $$ and $\omega$ is the normalized gauge form related to the equation.  
\end{thm}

We give the precise asymptotic formula for the example in \cite{DRS} (see Example \ref{example}).  

\begin{exa} If $f(x, y, z)=x^2+y^2 - h^2 z^2$ with $h\in \Bbb Z$ and  
$$ N(f, 1, T)= \# \{(\alpha, \beta, \gamma)\in \Bbb Z^3:  \ f(\alpha, \beta, \gamma) =1\ \text{with} \ \sqrt{\alpha^2+\beta^2+\gamma^2} \leq T \}  $$ for $T>0$, then  
$$ N(f, 1, T)\sim  \frac{8}{\pi \sqrt{1+h^2}} \cdot ( \prod_{p\mid h,\  p\neq 2} \frac{p-(\frac{-1}{p})}{p+1} )\cdot  T \log T $$ as $T\to \infty$. 
\end{exa}

The difficulty arising from this remaining case is that the infinity product 
 $$  \prod_{p \ \text{primes}} \alpha_p(f,a)  $$ is not convergent because 
 $\alpha_p(f,a)=1+p^{-1}$ for sufficiently large primes $p$ by \cite[Hilfssatz 12]{S}. Therefore the classical methods such as the circle method can not be applied. Instead, we use the method in \cite{WX} to study this problem, which is also valid over a general number field. In order to work over a general number field, one needs to extend the result of Oh and Shah in \cite{OS} which is only true over $\Bbb Q$, to a general number field. Indeed, such an extension over a general number field has been given in \cite{Zha} up to an implicit constant. One needs to determine this implicit constant in a precise way. Another ingredient of this paper comes from the fact that the stabilizer of a rational point is a split torus $\Bbb G_m$ when one considers the action of the spin group of $f$ on the equation. However,
the Tamagawa number of $\Bbb G_m$ is defined in a virtual way. One needs to use the Tamagawa number of $\Bbb G_m$ implicitly.  

Notations and terminology are standard if not explained. In particular,  we fix a number field $k$ and a non-degenerate ternary quadratic from $f(x,y,z)$ over $k$.
Let $a\in k^\times$ such that $-a\cdot det(f) \in (k^\times)^2$ and $$ X: \ \  f(x, y, z)=a$$  be an affine variety defined over $k$. Assume $X(k)\neq \emptyset$. Then $$-a\cdot det(f) \in (k^\times)^2 \Longrightarrow f \text{ is isotropic over $k$.} $$
Fixing $v_0\in X(k)$, one has     \begin{equation}\label{quot} X \cong  G /H  \ \ \ \text{with} \ \ \ G=\Spin(f) \cong  \SL_2 \ \ \ \text{and} \ \ \ H \cong \Bbb G_m \end{equation}  over $k$ by \cite[\S 5.6]{CTX}, where $H$ is the stabilizer of $v_0$ in $G$.

\begin{lem}\label{tran} For any field $F$ containing $k$, the group $G(F)$ acts on $X(F)$ transitively. \end{lem}
\begin{proof} Applying nonabelian Galois cohomology (see \cite[Chapter III, Prop.3.2.2 and \S 3.7]{Gi} ), one has the exact sequence of pointed sets
$$ 1\longrightarrow H(F) \longrightarrow G(F) \longrightarrow X(F) \longrightarrow H^1(F, H) \longrightarrow H^1(F, G) $$
by (\ref{quot}).   Since $H^1(F, \Bbb G_m)=0$ (Hilbert 90), one concludes that  $G(F)$ acts on $X(F)$ transitively.
\end{proof}

Since $k [X]^\times = k^\times$ by \cite[\S 5.6]{CTX}, one obtains that $X$ admits a unique $G$-invariant gauge form up to a
scalar by \cite[Chapter II, \S 2.2.2]{W}. Fixing an isomorphism $$H\cong \mathbb{G}_m=\Spec[t, t^{-1}], $$ we obtain the gauge form $\omega_H$ on $H$ by pulling back the differential form $t^{-1}dt$ on $\Bbb G_m$ through this isomorphism. Choose both a $G$-invariant gauge form $\omega_G$ on $G$ and a $G$-invariant gauge form $\omega_{X}$ on $X$ such that $\omega_G$, $\omega_H$ and $\omega_X$ match together algebraically in the sense of  \cite[\S 2.4]{W}. Let $\infty_k$ be the set of all archimedean primes of $k$. We use $v\leq \infty_k$ to denote all primes $v$ of $k$ and $v< \infty_k$ to denote all finite primes $v$ of $k$ respectively. Moreover, the completion of $k$ with respect to $v\leq \infty_k$ is denoted by $k_v$. Write ${\bf A}_{k}$ to be the adelic ring of $k$.  Let 
$ \lambda_v$,  $\mu_v$  and $ \nu_v $ be
the associated measures induced by $\omega_G$, $\omega_H$ and $\omega_X$ on $G(k_v)$,  $H(k_v)$ and  $X(k_v)$ for $v\leq \infty_k$ respectively.  Then
\begin{equation} \label{measures}  \int_{G(k_v)}  \phi (g) \lambda_v(g) = \int_{X(k_v)} \nu_v(\bar g) \int_{H(k_v)} \phi(gh) \mu_v(h)   \end{equation}
for any continuous function $\phi$ with compact support on $G(k_v)$ for $v\leq \infty_k$. Let $\parallel \cdot \parallel $ be the standard Euclidean norm on $\Bbb R^n$ and $\norm{\cdot}_{0}$ be a general linear norm on $\Bbb R^n$ satisfying 
    \begin{equation} \label{norm}
        C_0^{-1}\norm{\cdot} \leq \norm{\cdot}_0 \leq C_0 \norm{\cdot}.
    \end{equation}
for a fix constant $C_0>0$.

 \section{Counting in a single orbit}

In this section,  we will provide the precise asymptotic formula for counting the number of points in a single orbit by considering the action of an arithmetic lattice $\Gamma$ of $G(k)$ on $X(k)$. In fact, the general asymptotic formula has been established by the second author in \cite{Zha} up to an implicit constant. We will determine this implicit constant precisely.

\subsection{A model space}
Let $X_1$ be a variety over $k$ defined by the equation $x_{1,1}^2+x_{1,2}x_{2,1}=1$ and $G_1=\SL_2$. Define the action 
$$ G_1\times_k X_1 \longrightarrow X_1; \ \ \  (g,  \left[
    \begin{array}{cc}
      x_{1,1}   &  x_{1,2} \\
       x_{2,1}  &  -x_{1,1}
    \end{array}
    \right] ) \mapsto  g  \left[
    \begin{array}{cc}
      x_{1,1}   &  x_{1,2} \\
       x_{2,1}  &  -x_{1,1}
    \end{array}
    \right]  g^{-1} . $$ 
Let $$M_0=\left[
    \begin{array}{cc}
      1  &   \\
        &  -1
    \end{array}
    \right]\in X_1(k) \ \ \text{and $H_1$ be the stabilizer of $M_0$ in $G_1$. } $$ Then $H_1$ is a maximal torus of $G_1$. Fix an isomorphism $H_1\cong \Bbb G_m$.

  Consider the homomorphism of Lie groups 
  $$ \Nm: \ (k\otimes_{\Bbb Q} \R)^{\times} \cong (\bbR^\times)^r \oplus (\bbC^\times)^{s} \longrightarrow \R_{>0};  \ \ \ 
  (x_1, \cdots, x_r, y_1, \cdots, y_s) \mapsto \prod_{i=1}^r |x_i| \cdot \prod_{j=1}^s |y_j|^2 $$ where $r$  and $s$ are the number of real and complex primes of $k$ respectively. 
This induces the homomorphism of Lie groups
\begin{equation*}
   \begin{tikzcd}
      H_1(k\otimes_{\Bbb Q} \R) \cong (\bbR^\times)^r \oplus (\bbC^\times)^{s}   \arrow[r,"\Nm"] 
        & \bbR_{>0}
   \end{tikzcd}
\end{equation*}
with the kernel $\rmH_1^{\an}$. Then      
 \begin{equation*}
      H_1(k\otimes_{\Bbb Q} \bbR) = \rmH_1^{\an} \times  \rmH_1^{\spl} 
\end{equation*}
where    
$$\rmH_1^{\spl}= \{ 
    \left(
    \left[
    \begin{array}{cc}
        t &  \\
         & t^{-1}
    \end{array}\right],...,
    \left[
    \begin{array}{cc}
        t &  \\
         & t^{-1}
    \end{array}\right]
    \right) \in H_1(k\otimes_{\Bbb Q} \bbR) :  \ t\in \R_{>0} \} \cong \R_{>0} . $$

Choose the Haar measure $$(\bigwedge_{i=1}^r \frac{dx_i}{x_i})\wedge (\bigwedge_{j=1}^s2 \cdot \frac{du_j\wedge dv_j}{u_j^2+v_j^2}) \ \ \text{ on } \ \  (\bbR^\times)^r \oplus (\bbC^\times)^{s}$$ where $(x_1, \cdots, x_r)$ and $(u_1+v_1\sqrt{-1}, \cdots, u_s+v_s\sqrt{-1})$ are the coordinates of $(\bbR^\times)^r$ and $(\bbC^\times)^{s}$ respectively. 
By the fixed isomorphism $$H_1(k\otimes_{\Bbb Q} \R) \cong (\bbR^\times)^r \oplus (\bbC^\times)^{s} ,$$ one obtains the Haar measure $\mu_{H_1}$ on $H_1(k\otimes_{\Bbb Q} \R)$.  
Choosing the Haar measure $t^{-1}dt$ on $\R_{>0}$, one gets the Haar measure $\mu_{\rmH_1^{\spl}}$ on $\rmH_1^{\spl}$ by the fixed isomorphism $\rmH_1^{\spl}\cong \R_{>0}$. 
 There is a Haar measure $\mu_{\rmH_1^{\an}}$ on $\rmH_1^{\an}$ such that 
\begin{equation*}
    \mu_{H_1} \cong \mu_{\rmH_1^{\an}}\otimes \mu_{\rmH_1^{\spl}}.
\end{equation*}
By Lemma \ref{tran}, one can identify 
$$ X_1(k\otimes_{\Bbb Q} \R) \ \ \ \text{ with } \ \ \ G_1(k\otimes_{\Bbb Q} \R)/H_1(k\otimes_{\Bbb Q} \R)$$ and fix Haar measures $\lambda_{G_1}$ on $G_1(k\otimes_{\Bbb Q} \R)$ and $\nu_{X_1}$ on $ X_1(k\otimes_{\Bbb Q} \R)$ respectively such that $(\lambda_{G_1}, \mu_{H_1}, \nu_{X_1})$ match up.
   
 \begin{prop}\label{propSec2ModelCase} If $\Gamma_1$ is an arithmetic lattice in $G_1(k)$, then
\begin{equation*}
     \# \left\{
    M \in \Gamma_1 \cdot M_0: \ 
    \norm{M}_{0} \leq T
    \right\}
    \sim
     \frac{\mu_{\rmH_1^{\an}}{(\rmH_1^{\an}/\rmH_1^{\an}\cap \Gamma_1})}
     {\lambda_{G_1}{(G_1(k\otimes_{\Bbb Q} \bbR)/\Gamma_1)}}\cdot 
     \log{T}  \int_{B_T} \nu_{X_1}
\end{equation*}
as $T\to +\infty$, where 
$$ B_T = \left\{ x\in X_1(k\otimes_{\Bbb Q} \bbR): \  \norm{x}_{0}\leq T
\right\} $$ and $\norm{\cdot}_{0}$ is a linear norm on $(k\otimes_{\Bbb Q} \bbR)^3$ satisfying (\ref{norm}). 
\end{prop}
   
 \begin{proof} Consider 
 \begin{equation} \label{identification}
    G_1(k\otimes_{\Bbb Q}  \bbR) \cong \SL_2(\bbR) ^{\oplus r} \bigoplus \SL_2(\bbC)^{\oplus s} \ \ \ \text{acting on} \ \ \ (k\otimes_{\Bbb Q} \bbR)^2 \cong (\bbR^2) ^{\oplus r} \bigoplus (\bbC^2)^{\oplus s}
\end{equation}
as an $\bbR$-vector space naturally. This action induces an action of $G_1(k\otimes_{\Bbb Q}  \bbR)$ on 
$$ \bigwedge^{r+2s} (k\otimes_{\Bbb Q} \bbR)^2 $$
as an $\bbR$-vector space.

For any $v \in \R^2$, we define $v^{i}  \in (k\otimes_{\Bbb Q} \bbR)^2$ such that the $i$-th component of $v^{i}$ is $v$ and the rest is 0 by identification (\ref{identification}) for $1\leq i\leq r$.
Similarly  for any $w \in \bbC^2$, we define $w^{r+j}\in (k\otimes_{\Bbb Q} \bbR)^2$ such that the $(r+j)$-th component is $w$ and the rest is 0. Let $ \{e_1=(1,0), \ e_2=(0,1)\}$ be the standard basis of $\bbR^2$. Write 
\begin{equation*}
    \Lambda (e_k):= (\bigwedge_{i=1}^r e^i_k) \wedge (\bigwedge_{j=1}^s e^{r+j}_{k})\wedge (\bigwedge_{j=1}^s\sqrt{-1}e^{r+j}_{k}) \in  \bigwedge^{r+2s} (k\otimes_{\Bbb Q} \bbR)^2 
\end{equation*}
for $k=1,2$. For $g\in G_1(k\otimes_{\Bbb Q} \bbR) $ and $\ep>0$, we define 
\begin{equation*}
    \Omega_{g,\ep}:=
    \left\{
    t \in \rmH_1^{\spl}: \ 
    \norm{g t \cdot \Lambda (e_k)} \geq \ep, \;k=1,2
    \right\}.
\end{equation*}
By the equidistribution in \cite[Theorem 1.2]{Zha} (cf. \cite{OS}), one has,  for any $\ep>0$, 

\begin{equation} \label{propPrimCount}
    \# \left\{
    M \in \Gamma_1 \cdot M_0: \ 
    \norm{M}_{0} \leq T
    \right\}
    \sim
     \frac{\mu_{\rmH_1^{\an}}{(\rmH_1^{\an}/\rmH_1^{\an}\cap \Gamma_1})}
     {\lambda_{G_1}{(G_1(k\otimes_{\Bbb Q} \bbR)/\Gamma_1)}} \cdot
     \int_{\norm{gM_0}_{0}\leq T} \mu_{\rmH_1^{\spl}}(\Omega_{g,\ep}) \nu_{X_1}
\end{equation}
as $T\to +\infty$.

Let $$\rmK:= \SO_2(\R)^{\oplus r} \bigoplus \SU_2(\R)^{\oplus s}$$ be a maximal compact subgroup of $G_1(k\otimes_{\Bbb Q} \bbR)$
and 
\begin{equation}\label{cou}
    \rmU=\left\{
    \left(
    \left[
    \begin{array}{cc}
       1  & x_i \\
       0  & 1
    \end{array}
    \right],
    \left[
    \begin{array}{cc}
       1  & y_j \\
       0  & 1
    \end{array}
    \right]
    \right)
    : \  x_i \in \bbR, 1\leq i\leq r; \ y_j \in \bbC, 1\leq j\leq s 
    \right\}
\end{equation}
Then  the following map given by multiplication
\begin{equation} \label{iwasawa}
     \rmK \times \rmU \times H_1(k\otimes_{\Bbb Q} \R) \to
     G_1(k\otimes_{\Bbb Q} \bbR) 
\end{equation}
is  surjective such that all fibres are torsors under a finite group $\rmK\cap H_1(k\otimes_{\Bbb Q} \R)$ by Iwasawa decomposition. If $g=kuh$ with $k\in \rmK$, $u\in \rmU$ and $h\in H_1(k\otimes_{\Bbb Q} \R)$ under the decomposition (\ref{iwasawa}), then
\begin{equation*}
    \mu_{\rmH_1^{\spl}}(\Omega_{g,\ep}) = \mu_{\rmH_1^{\spl}}(\Omega_{u,\ep}).
\end{equation*}
If 
$$ u=  \left(\left[
    \begin{array}{cc}
       1  & x_1 \\
       0  & 1
    \end{array}
    \right], \cdots, 
    \left[
    \begin{array}{cc}
       1  & x_r \\
       0  & 1
    \end{array}
    \right],
    \left[
    \begin{array}{cc}
       1  & y_1 \\
       0  & 1
    \end{array}
    \right], \cdots,  \left[
    \begin{array}{cc}
       1  & y_s \\
       0  & 1
    \end{array}
    \right] \right) \in \rmU ,$$
then 
\begin{equation}\label{equaVolPolytope}
 \log: \ \    \Omega_{u,\ep} \cong 
    \left[\frac{\log{\ep}}{n}, 
    -\frac{\log{\ep}}{n} + \frac{1}{2n} \left(
    \sum \log{(x_i^2+1)} + \sum 2\log{(|y_j|^2+1)}
    \right)
    \right]
\end{equation}
by taking logarithm with the following identities 
\begin{equation*}
\begin{aligned}
       & \diag(\exp(t),\exp(-t)) \cdot \Lambda (e_1) = \exp{((r+2s)t)} \Lambda (e_1) \\
       & \diag(\exp(t),\exp(-t)) \cdot \Lambda (e_2) = \exp{(-(r+2s)t)} \Lambda (e_2)
\end{aligned}
\end{equation*}
for $t\in (k\otimes_{\Bbb Q} \bbR)$, $u \cdot \Lambda (e_1)=\Lambda (e_1)$ and 
\begin{equation*}
    \norm{u\cdot \Lambda (e_2)}  
    = \norm{
    \left( \begin{array}{c}
         x_1 \\
         1
    \end{array}\right)^1 \wedge ...\wedge
    \left( \begin{array}{c}
         y_s \\
         1
    \end{array}\right)^{r+s}
    } = \prod(x_i^2+1)^{1/2} \cdot \prod (|y_j|^2+1).
\end{equation*}

Fix a Haar measures $\mu_{\rmK}$ on $\rmK$ such that $\mu_{\rmK}(\rmK)=1$.  Choose a Haar measure $\mu_{\rmU}$ on $\rmU$ such that  $\mu_{\rmK}\otimes \mu_{\rmU} \otimes \mu_{H_1}$ is pushed to $\lambda_{G_1}$ by the surjective map (\ref{iwasawa}). Therefore $\nu_{X_1}$ is identified with (the push-forward of) $\mu_{\rmK}\otimes \mu_{\rmU}$. 
Since  $$\rmU \cong k\otimes_{\Bbb Q} \bbR\cong \bbR^r \oplus \bbC^s , $$  there exists a constant $C_U>0$ such that $\mu_{\rmU}$ is $C_U$ times the standard Lebesgue measure on $\bbR^r \oplus \bbC^s$. Moreover, we can identify $B_T$ as a subset of $\rmK \times \rmU$ by (\ref{iwasawa}).

For $T>0$ and $\eta\in (0,1)$, we define another subset $B_{T,\eta}$ of $\rmK \times \rmU$ by 
\begin{equation*}
    B_{T,\eta}:= \left\{(k,u)\in B_T : \ 
    |x_i| \geq \eta T,\  1\leq i\leq r; \ |y_j | \geq \eta T, \  1\leq j\leq s   \right\}
\end{equation*} by using the coordinates (\ref{cou}). 
When $\norm{\cdot}_0= \norm{\cdot}$, we write $B_T^{\st}$ and $B^{\st}_{T,\eta}$ for the corresponding sets respectively.
Then 
\begin{equation*}
    B^{\st}_{T/C_{0}} \subset B_T \subset B^{\st}_{C_{0}T}
\end{equation*}
by (\ref{norm}). 

Since
\begin{equation*}
    \left[
    \begin{array}{cc}
       1  &  x\\
       0  & 1
    \end{array}
    \right] 
     \left[
    \begin{array}{cc}
       1  &  0\\
       0  & -1
    \end{array}
    \right]
     \left[
    \begin{array}{cc}
       1  &  x\\
       0  & 1
    \end{array}
    \right]^{-1}
    =  \left[
    \begin{array}{cc}
       1  &  -2x\\
       0  & -1
    \end{array}
    \right]
\end{equation*}  and  $B_T^{\st}$ is left $\rmK$-invariant, 
we can write 
\begin{equation*}
\begin{aligned}
        B_T^{\st} &\cong\rmK \times 
    \left\{
    (x_i,y_j) : \
    2(r+s) + 4\sum |x_i|^2 + 4 \sum |y_j|^2 \leq T^2
    \right\}\\
    &= \rmK \times 
    \left\{
    (x_i,y_j) : \ 
    \sum |x_i|^2 + \sum |y_j|^2 \leq \frac{1}{4}(T^2-2r-2s)
    \right\}.
\end{aligned}
\end{equation*}
By (\ref{equaVolPolytope}), we find that 
\begin{equation*}
    \log (\Omega_{g,\ep}) \subset 
    \left[
    \frac{\log{\ep}}{n}, -\frac{\log{\ep}}{n}+\log{T}
    \right]
\end{equation*} for $[g]\in B_T^{\st}$. 
On the other hand if $[g]$ is in $B_{T,\eta}$, then
\begin{equation*}
    \log (\Omega_{g,\ep}) \supset 
    \left[
    \frac{\log{\ep}}{n}, -\frac{\log{\ep}}{n}+\log{T}+ \log{\eta}
    \right].
\end{equation*}
Consequently, one has 
\begin{equation*}
    \mu_{\rmH_1^{\spl}}(\Omega_{g,\ep}) \leq \log{T} + \log{C_{0}}+ \frac{2\log{\ep}}{n}
\end{equation*}
for $[g] \in B_T \subset B^{\st}_{C_{0}T}$ and 
\begin{equation*}
    \log{T} +\frac{2\log{\ep}}{n} + \log{\eta}
    \leq \mu_{\rmH_1^{\spl}}(\Omega_{g,\ep}) 
    \leq \log{T} + \log{C_{0}}+ \frac{2\log{\ep}}{n}
\end{equation*} for $[g] \in B_{T,\eta}$. 
Asymptotically, this is just $\log{T}$.

It remains to show that the complement of $B_{T,\eta}$ is negligible in $B_T$.
Indeed, it is not hard to see from the definition that there exists $C_1>0$ (independent of $\eta$) such that 
\begin{equation*}
     \nu_{X_1} (B^{\st}_T \setminus B^{\st}_{T,\eta}) \leq ( C_1 \eta ) \cdot
     \nu_{X_1} (B^{\st}_T).
\end{equation*}
We may also require $C_1$ to satisfy
\begin{equation*}
     \nu_{X_1} (B^{\st}_{C_{0}T}) \leq C_1 \cdot  \nu_{X_1} (B^{\st}_{C_{0}^{-1}T}) .
\end{equation*}
Since $$ (B_T\setminus B_{T,\eta} )\subset (B_{C_{0}T} \setminus B_{C_{0}T,C_{0}^{-1}\eta}) , $$ we have 
\begin{equation*}
     \nu_{X_1} (B_T \setminus B_{T,\eta}) 
    \leq (C_1^2 C_{0}^{-1}\eta )\cdot
     \nu_{X_1} (B_T).
\end{equation*}
By firstly letting $T\to +\infty$ and then letting $\eta\to 0$, we conclude that  for any $\ep>0$
\begin{equation*}
    \int_{B_T} \mu_{\rmH_1^{\spl}}(\Omega_{g,\ep}) \nu_{X_1}
    \sim 
      \log{T} \int_{B_T} \nu_{X_1} 
\end{equation*}
as $T\to +\infty$. The result follows from combining (\ref{propPrimCount}).
 \end{proof}

\subsection{The general case} For general case (\ref{quot}), we fix an isomorphism $H\cong \Bbb G_m$ which induces 
the homomorphism of Lie groups
\begin{equation} \label{prod}
\begin{aligned}
   \begin{tikzcd}
     H(k\otimes_{\Bbb Q} \R) \cong   (\bbR^\times)^r \oplus (\bbC^\times)^{s}   \arrow[r,"\Nm"] 
        & \bbR_{>0}; 
   \end{tikzcd} \\
    (x_1, \cdots, x_r, y_1, \cdots, y_s) \mapsto \prod_{i=1}^r |x_i| \cdot \prod_{j=1}^s |y_j|^2
    \end{aligned}
\end{equation}
with the kernel $\rmH^{\an}$ as before. Similarly, one has     
     $$ H(k\otimes_{\Bbb Q} \bbR) = \rmH^{\an} \times  \rmH^{\spl}  \ \ \ \text{
with  } \ \ \ \rmH^{\spl} \cong \R_{>0} . $$
Choosing the Haar measure $\mu_{\rmH^{\spl}}$ on $\rmH^{\spl}$ induced by $t^{-1}dt$ on $\R_{>0}$, one obtains a Haar measure $\mu_{\rmH^{\an}}$ on $\rmH^{\an}$ such that 
$\mu_{\infty_k} \cong \mu_{\rmH^{\an}}\otimes \mu_{\rmH^{\spl}}$, where $\mu_{\infty_k}$ is given by (\ref{measures}).
Now we can extend \cite[Theorem 1.2]{OS} over $\bbQ$ to a general number field. 

\begin{thm}\label{propMainSec2}
If $\Gamma$ is an arithmetic lattice in $G(k)$, then
\begin{equation*}
     \# \left\{
    v \in \Gamma \cdot v_0: \ 
    \norm{v}_{0} \leq T
    \right\}
    \sim
     \frac{\mu_{\rmH^{\an}}{(\rmH^{\an}/\rmH^{\an}\cap \Gamma})}
     {\lambda_{\infty_k}{(G(k\otimes_{\Bbb Q} \bbR)/\Gamma)}}\cdot 
     \log{T}  \int_{B_T} \nu_{\infty_k}
\end{equation*}
as $T\to +\infty$, where $\lambda_{\infty_k}$ and $\nu_{\infty_k}$ are given by (\ref{measures}), 
$$ B_T = \left\{ x\in X(k\otimes_{\Bbb Q} \bbR): \  \norm{x}_{0}\leq T
\right\} $$ and $\norm{\cdot}_{0}$ is a linear norm on $(k\otimes_{\Bbb Q} \bbR)^3$ satisfying (\ref{norm}).
\end{thm}

\begin{proof} Let $det(f)=\delta$. Then $-\delta f$ is also an isotropic quadratic form over $k$ with the same discriminant as that of $$f_1(x,y,z):= x^2+yz \ \ \ \text{ up to $(k^\times)^2$} .$$ This implies that the quadratic space $(k^3, -\delta f)$ is isometric to the quadratic space $(k^3, f_1)$ over $k$ by \cite[58:4 Theorem and 58:6]{OM}. There is an invertible $k$-linear map $\sigma$ from $k^3$ to $k^3$ such that $\sigma$ induces an isomorphism of algebraic groups 
$$ \Phi: G=\Spin(f)=\Spin(-\delta f) \longrightarrow G_1 $$ over $k$.  Since any two maximal split tori of $G_1$ over $k$ are conjugated by an element in $G_1(k)$, there is $g\in G_1(k)$ such that the following diagram of algebraic groups 
$$\xymatrixcolsep{5pc}\xymatrix{ H \ar[d] \ar[r]^{g \circ \Phi|_{H}}_{\cong} &  H_1 \ar[d] \\
G \ar[r]^{g \circ \Phi}_{\cong} & G_1}$$
commutes, where the vertical arrows are the inclusion maps. Therefore $g\circ \Phi$ induces 
$$ \overline{g\circ \Phi}: \ X \longrightarrow X_1 $$ 
an isomorphism of homogeneous spaces over $k$ by sending $v_0$ to $M_0$. The conjugation of $g$ induces an invertible $k$-linear map on $k^3$ given by 
$$ g: \ k^3\longrightarrow k^3 ;  \ \ \ \left[
    \begin{array}{cc}
      x   &  y \\
       z  &  -x
    \end{array}
    \right]  \mapsto  g  \left[
    \begin{array}{cc}
      x   &  y \\
       z  &  -x
    \end{array}
    \right]  g^{-1} $$
for any $(x,y,z)\in k^3$.  By considering $k$-points, one has the following commutative diagram 
$$\xymatrixcolsep{5pc}\xymatrix{ X(k)  \ar[d] \ar[r]^{\overline{g \circ \Phi}}_{\cong} &  X_1(k) \ar[d] \\
k^3 \ar[r]^{g \circ \sigma}_{\cong} & k^3}$$
where the vertical arrows are the inclusion maps of the chosen coordinates. 

Let $\Gamma_1$ be the image of $\Gamma$ under the map $g \circ \Phi$ and $\norm{\cdot}_1$ be
the linear norm on $(k\otimes_{\Bbb Q} \bbR)^3$ defined by $\norm{\cdot}_1:= \norm{g\circ \sigma (\cdot)}_0$. Since  $\overline{g \circ \Phi}$ induces the bijection between the following sets
\begin{equation*}
    \begin{aligned}
    \left\{
    v\in \Gamma \cdot v_0: \ 
    \norm{v}_0\leq T
    \right\} 
    & \longleftrightarrow
    \left\{
    M\in \Gamma_1 \cdot M_0: \ 
    \norm{M}_1\leq T
    \right\}
    \end{aligned}
\end{equation*}
and the Haar measures are preserved,  one concludes that Theorem \ref{propMainSec2} follows from Proposition \ref{propSec2ModelCase}.
\end{proof}

 \section{Main result}

Let $\frak o_k$ be the ring of integers of $k$ and $\mathcal X$ be a separated scheme of finite type over $\frak o_k$ such that 
$ \mathcal X\times_{\frak o_k}k = X$. Since $H\cong  \Bbb G_m$ over $k$, there is the unique maximal open compact subgroup $N_v$ of $H(k_v)$ for all $v<\infty_k$.  Write $N_v=H(k_v)$ for $v\in \infty_k$.

Let $\mathcal G$ and $\mathcal H$ be smooth connected group schemes over $\frak o_k$ such that 
$$\mathcal G\times_{\frak o_k} k=G \ \ \ \text{and} \ \ \ \mathcal H \times_{\frak o_k} k=H$$ respectively.  There is a finite set $S$ of primes of $k$ containing $\infty_k$ such that 
 \begin{equation} \label{int-mod} \mathcal X\times_{\frak o_k} \frak o_{k_v} \cong (\mathcal G\times_{\frak o_k}\frak o_{k_v})/(\mathcal H \times_{\frak o_k} \frak o_{k_v}) 
 \ \ \ \text{and} \ \ \ \mathcal H(\frak o_{k_v})=N_v \end{equation} for all primes $v\not\in S$ induced by ($\ref{quot}$), where $\frak o_{k_v}$ is the integral ring of $k_v$. 
\begin{equation}\label{tama} h_k=[ H({\bf A}_{k}) : H(k) \cdot  \prod_{v\leq \infty_k} N_v]  \ \ \ \text{and} \ \ \ H(k) \cap  \prod_{v \leq \infty_k} N_v \cong \frak o_{k}^\times  \end{equation} 
where  $h_k$ is the class number of $k$.

Fix a maximal open compact subgroup $M_v$ of $G(k_v)$ containing $N_v$ for $v\in (S\setminus \infty_k)$.  Define 
 $$ \Gamma_v = \begin{cases} \mathcal G(\frak o_{k_v})  \ \ \ & \text{$v\not\in S$} \\
 \{g\in M_v: \ g \cdot \mathcal X(\frak o_{k_v})=\mathcal X(\frak o_{k_v}) \} \ \ \ & \text{$v\in (S\setminus \infty_k)$} \\ 
 G(k_v) \ \ \ & \text{$v\in\infty_k$}. 
 \end{cases} $$ Then $\Gamma_v$ acts on $\mathcal X (\frak o_{k_v})$ such that the number of orbits $[\mathcal X(\frak o_{k_v}): \Gamma_v]$ is finite for all $v\leq \infty_k$. Moreover, one has that $N_v \subset \Gamma_v$ for all $v\leq \infty_k$.

 \begin{lem} \label{loc} There are finitely many $x_1, \cdots, x_l \in \mathcal X(\frak o_k)$ such that 
 $$ \mathcal X(\frak o_k) = \bigcup_{i=1}^l [ (\prod_{ v \leq \infty_k} \Gamma_v )\cdot x_i\cap   \mathcal X(\frak o_k) ] \ \ \ \text{ with } \ \ \  l= \prod_{v\in (S\setminus \infty_k)} [\mathcal X(\frak o_{k_v}): \Gamma_v] .$$ \end{lem}
 \begin{proof} By (\ref{int-mod}) and \cite[Chapter III, Proposition 3.2.2]{Gi}, the following sequence of pointed sets $$ 1\rightarrow {\mathcal H}(\frak o_{k_v}) \rightarrow {\mathcal G} (\frak o_{k_v}) \rightarrow {\mathcal X}(\frak o_{k_v}) \rightarrow \check{H}^1(\frak o_{k_v}, {\mathcal H}) $$ is exact for $v\not\in S$. Since $\mathcal H$ is connected and smooth over $\frak o_{k_v}$, every torsor over $\frak o_{k_v}$ under $\mathcal H$ is trivial by Hensel's Lemma (see \cite[Chapter 3, Thm.3.11]{PR94}) and Lang's Theorem (see \cite[Chapter 6, Thm.6.1]{PR94}). This implies $\check{H}^1(\frak o_{k_v}, {\mathcal H})$ is trivial. Therefore 
$$ [\mathcal X(\frak o_{k_v}): \Gamma_v]=1 $$ for all primes $v\not\in S$.

Since $X$ satisfies strong approximation off $\infty_k$ by \cite[Theorem 3.7 and \S 5.6] {CTX}, one obtains 
 $$ (\prod_{ v \leq \infty_k} \Gamma_v ) \cdot  \mathcal X(\frak o_k) = \prod_{v \leq \infty_k}  \mathcal X(\frak o_{k_v})  $$ by Lemma \ref{tran}.
 This implies that 
 $$ l= \prod_{v\leq \infty_k} [\mathcal X(\frak o_{k_v}): \Gamma_v] = \prod_{v\in (S\setminus \infty_k)} [\mathcal X(\frak o_{k_v}): \Gamma_v] $$ as desired.
 \end{proof}
 
 Let $$\Gamma= G(k) \cap (\prod_{ v \leq \infty_k} \Gamma_v ) . $$  Then $\Gamma$  is an arithmetic lattice in $G(k)$ and acts on $\mathcal X(\frak o_k)$. 
 
 \begin{rem} By \cite[6.9 Theorem]{BHC}, the number of  $\Gamma$-orbits of $\mathcal X(\frak o_k)$ is finite when $k=\Bbb Q$. It should be pointed out that a simple proof of this finiteness in \cite[Corollary 2.5]{WX} is not correct because the natural map in the proof is not injective. However, this finiteness result plays no role in the proof of the main results in \cite{WX}. In fact, the precise descriptions of number of orbits are given by \cite[Prop. 2.9, Prop.2.11 and Prop.2.12]{WX} and in particular the finiteness of orbits follows in those cases over a general number field.   
 \end{rem}

 Since $\Gamma$ acts on $[ ((\prod_{ v \leq \infty_k} \Gamma_v )\cdot x) \cap   \mathcal X(\frak o_k)] $ for any $x\in \mathcal X(\frak o_k)$, one can give the precise description of  number of orbits. 
 
 \begin{lem}  \label{glo} If $x\in \mathcal X(\frak o_k)$, then there are $y_1, \cdots, y_{h_x}$ in $\mathcal X(\frak o_k)$ such that 
 $$ ((\prod_{ v\leq \infty_k} \Gamma_v )\cdot x )\cap   \mathcal X(\frak o_k) = \bigcup_{i=1}^{h_x} \Gamma y_i $$
 with
 $$ h_x= [H_x({\bf A}_{k}): H_x(k)\cdot \prod_{v \leq \infty_k}(H_x(k_v)\cap  \Gamma_v)] <\infty $$ where $H_x$ is the stabilizer of $x$ in $G$.
 \end{lem}  
 \begin{proof}  For any $$y\in  ((\prod_{ v\leq \infty_k} \Gamma_v )\cdot x )\cap   \mathcal X(\frak o_k) , $$ there is $\sigma \in \prod_{ v\leq \infty_k} \Gamma_v$ such that $y=\sigma x$.    By Lemma \ref{tran}, there is $\tau\in G(k)$ such that $y=\tau x$. This implies that $\tau^{-1} \sigma \in H_x({\bf A}_{k})$. Define a map 
 $$ \Phi:  \ ( (\prod_{ v\leq \infty_k} \Gamma_v )\cdot x )\cap   \mathcal X(\frak o_k) \longrightarrow H_x({\bf A}_{k})/[H_x(k)\cdot \prod_{v \leq \infty_k}(H_x(k_v)\cap  \Gamma_v )] $$ by sending $y$ to $\tau^{-1} \sigma$. Suppose there are $$\sigma'\in  \prod_{ v\leq \infty_k} \Gamma_v \ \ \ \text{ and } \ \ \ \tau'\in  G(k)$$ such that $y=\sigma' x=\tau' x$. Then $$\sigma^{-1} \sigma' \in  \prod_{v \leq \infty_k}(H_x(k_v)\cap  \Gamma_v )  \ \ \ \text{and} \ \ \ \tau^{-1} \tau' \in H_x(k) . $$ 
 This implies that the map $\Phi$ is well-defined. For any $\xi \in H_x({\bf A}_{k})\subset G({\bf A}_{k})$, there are $$\gamma \in \prod_{ v\leq \infty_k} \Gamma_v \ \ \ \text{ and } \ \ \  \eta \in G(k)$$ such that $\xi=\eta\cdot \gamma$ by strong approximation for $G$ (see \cite[Theorem 7.12]{PR94}). Therefore 
 $$ \eta^{-1}x=\gamma x \in ( (\prod_{ v\leq \infty_k} \Gamma_v )\cdot x )\cap X(k) =( (\prod_{ v\leq \infty_k} \Gamma_v )\cdot x )\cap   \mathcal X(\frak o_k)  $$
 and $\Phi$ is surjective. 
  Suppose $$\Phi(y)= \Phi (y_1) \ \ \ \text{ with } \ \ y, y_1\in  ( (\prod_{ v\leq \infty_k} \Gamma_v )\cdot x )\cap   \mathcal X(\frak o_k) . $$ There are $$\sigma, \sigma_1 \in  \prod_{ v\leq \infty_k} \Gamma_v \ \ \ \text{ and  } \ \ \ \tau, \tau_1 \in G(k)$$ with $y=\sigma x= \tau x$ and $y_1=\sigma_1 x =\tau_1 x$ respectively such that 
 $$( \tau^{-1} \sigma )^{-1} (\tau_1^{-1} \sigma_1) \in H_x(k)\cdot \prod_{v \leq \infty_k}(H_x(k_v)\cap  \Gamma_v ) . $$
This implies that there are $h\in H_x(k)$ and $\delta \in \prod_{v \leq \infty_k}(H_x(k_v)\cap  \Gamma_v )  $ such that 
$$ \tau_1\tau^{-1} h= \sigma_1\sigma^{-1} \delta \in (G(k) \cap \prod_{v \leq \infty_k} \Gamma_v) = \Gamma . $$ Then $y_1\in \Gamma y$. 
Therefore the map induced by $\Phi$ 
 $$   [(\prod_{ v\leq \infty_k} \Gamma_v )\cdot x \cap   \mathcal X(\frak o_k)]/\Gamma  \xrightarrow{\cong}  H_x({\bf A}_k)/[H_x(k)\cdot \prod_{v \leq \infty_k}(H_x(k_v)\cap  \Gamma_v )] $$ is a bijection.  The finiteness of $h_x$ follows from \cite[Chapter 5, Thm.5.1]{PR94}. 
  \end{proof}

Define 
$$N({\mathcal X}, T)= \# \{(\alpha, \beta, \gamma) \in {\mathcal  X}(\frak o_k)\subset X(k): \ \  \norm{ (\alpha, \beta, \gamma)}_0   \leq T \} $$ 
for $T> 0$, where $\norm{\cdot}_0$ is a norm on $(k\otimes_{\Bbb Q} \Bbb R)^3$ satisfying (\ref{norm}). Let $r$ and $s$ be the numbers of real and complex primes of $k$ respectively, 
$w_k$ be the number of roots of unity in $k$, $d_k$ be the discriminant of $k$ and $R_k$ be the regulator of $k$. The main result of this section is the following theorem.

\begin{thm}\label{main} If $q_v$ is the number of elements in the residue field of $k$ at $v<\infty_k$, then
$$ N({\mathcal X}, T) \sim  \frac{2^r (2\pi)^s}{w_k\sqrt{|d_k|}} R_k h_k \cdot (\prod_{v< \infty_k} (1-q_v^{-1}) \int_{\mathcal X(\frak o_{k_v})} \nu_v) \cdot  (\log{T}  \int_{B_{T}} \nu_{\infty_k})$$ as $T\to \infty$ where 
$$ B_T= \{ (\alpha, \beta, \gamma) \in X(k\otimes_{\Bbb Q} \Bbb R): \ \ \norm{ (\alpha, \beta, \gamma) }_0 \leq T \} . $$ 
\end{thm}  
\begin{proof} Since $\mathcal X(\frak o_k)$ can be decomposed into the following disjoint union 
$$ \mathcal X(\frak o_k)= \bigcup_{i=1}^l [ (\prod_{v \leq \infty_k} \Gamma_v )\cdot x_i\cap   \mathcal X(\frak o_k) ]= \bigcup_{i=1}^l \bigcup_{j=1}^{h_{x_i}} \Gamma y_{i, j} $$
with 
$$ (\prod_{ v\leq \infty_k} \Gamma_v )\cdot x_i \cap   \mathcal X(\frak o_k) = \bigcup_{j=1}^{h_{x_i}} \Gamma y_{i, j} \ \ \ \text{and} \ \ \ h_{x_i}= [H_{x_i}({\bf A}_{k}): H_{x_i}(k)\cdot \prod_{v \leq \infty_k}(H_{x_i}(k_v)\cap  \Gamma_v )] $$ for $1\leq i\leq l$ by Lemma \ref{loc} and Lemma \ref{glo}, one obtains that 
$$N({\mathcal X}, T) = \sum_{i=1}^l \sum_{j=1}^{h_{x_i}} N(y_{i,j}, T) $$ where 
$$N(y_{i,j}, T)= \{ (\alpha, \beta, \gamma)\in \Gamma y_{i,j}: \ \ \norm{ (\alpha, \beta, \gamma) }_0 \leq T \} $$ for $1\leq j\leq h_{x_i}$ and $1\leq i\leq l$. 
 By Theorem \ref{propMainSec2}, one has  
$$N(y_{i,j}, T) \sim  \frac{\mu_{\rmH_{y_{i,j}}^{\an}}{(\rmH_{y_{i,j}}^{\an}/\rmH_{y_{i,j}}^{\an}\cap \Gamma})}
     {\lambda_{\infty_k}{(G(k\otimes_{\Bbb Q} \bbR)/\Gamma)}}\cdot 
     \log{T}  \int_{B_{i, j}(T)} \nu_{\infty_k} $$  with 
$$ B_{i, j}(T) = \{ z\in  G(k\otimes_{\Bbb Q}\Bbb R) y_{i,j} \subseteq X(k\otimes_{\Bbb Q} \Bbb R):   \  \norm{ z }_0 \leq T \} $$ as $T\to \infty$  for $1\leq j\leq h_{x_i}$ and $1\leq i\leq l$.

By Lemma \ref{tran}, there is $g_{i, j}\in G(k)$ such that 
\begin{equation} \label{transv-y}   g_{i, j} v_0= y_{i, j} \ \ \ \text{ and } \ \ \ H_{y_{i, j}}=g_{i, j} H g_{i, j}^{-1} \end{equation}  for $1\leq j\leq h_{x_i}$ and $1\leq i\leq l$. Since $H(k_v)\cap g_{i, j}^{-1}\Gamma_v g_{i, j}$ is an open and compact subgroup of $H(k_v)$ for all $v<\infty_k$, one has  
\begin{equation} \label{loc-max}  (H(k_v)\cap g_{i, j}^{-1}\Gamma_v g_{i, j}) \subseteq N_v \ \ \ \text{and} \ \ \ H(k_v)\cap g_{i, j}^{-1}\Gamma_v g_{i, j}=N_v\cap g_{i, j}^{-1}\Gamma_v g_{i, j}  \end{equation} 
for all $v<\infty_k$, $1\leq j\leq h_{x_i}$ and $1\leq i\leq l$. Then 
\begin{equation} \label{inner-tamagawa-y}  \begin{aligned} & [(H(k) \cap  \prod_{v\leq \infty_k} N_v) : (H(k)\cap \prod_{v\leq \infty_k} ( H(k_v)\cap g_{i, j}^{-1}\Gamma_v g_{i, j}))]^{-1} \mu_{\rmH_{y_{i, j}}^{\an}}{(\rmH_{y_{i, j}}^{\an}/\rmH_{y_{i, j}}^{\an}\cap \Gamma}) \\
= \ & [(H(k) \cap  \prod_{v\leq \infty_k} N_v) : (H(k)\cap \prod_{v\leq \infty_k} ( H(k_v)\cap g_{i, j}^{-1}\Gamma_v g_{i, j}))]^{-1} \mu_{\rmH^{\an}}{(\rmH^{\an}/\rmH^{\an}\cap g_{i, j}^{-1}\Gamma g_{i, j}}) \\
= \ & \mu_{\rmH^{\an}}{(\rmH^{\an}/\rmH^{\an}\cap (H(k)\cap \prod_{v\leq \infty_k} N_v)})= \frac{2^r (2\pi)^s}{w_k} R_k 
\end{aligned} \end{equation} by the product formula (see \cite[Chapter III, (1.3) Proposition]{N})  and (\ref{prod}) and the computation in \cite[P.140]{Vo}.

 By Lemma \ref{tran}, there is $g_i\in G(k)$ such that 
 \begin{equation} \label{transv-x} g_i v_0= x_i \ \ \ \text{ and } \ \ \ H_{x_i}=g_i H g_i^{-1} \end{equation} for $1\leq i\leq l$. By the same argument as (\ref{loc-max}) and (\ref{inner-tamagawa-y}), one obtains that 
\begin{equation} \label{inner-tamagawa-x}  \begin{aligned} & [(H(k) \cap  \prod_{v\leq \infty_k} N_v) : (H(k)\cap \prod_{v\leq \infty_k} ( H(k_v)\cap g_{i}^{-1}\Gamma_v g_{i}))]^{-1} \mu_{\rmH_{x_{i}}^{\an}}{(\rmH_{x_{i}}^{\an}/\rmH_{x_{i}}^{\an}\cap \Gamma}) \\
= \ & \mu_{\rmH^{\an}}{(\rmH^{\an}/\rmH^{\an}\cap (H(k)\cap \prod_{v\leq \infty_k} N_v)})= \frac{2^r (2\pi)^s}{w_k} R_k 
\end{aligned} \end{equation}
for $1\leq i\leq l$.

For each $1\leq j\leq h_{x_i}$, there is $$(\gamma_v)_{v\leq \infty_k}\in \prod_{v\leq \infty_k} \Gamma_v \ \ \ \text{such that} \ \ \  (\gamma_v)_{v\leq \infty_k} \cdot x_i = y_{i, j} $$
by the partition of $\mathcal X(\frak o_k)$. One can write 
$$ h_v = g_{i, j}^{-1} \gamma_v  g_i \in H (k_v) $$ for all $v\leq \infty_k$ by (\ref{transv-y}) and (\ref{transv-x}). Since $H(k_v)$ is commutative, one obtains that 
$$H(k_v)\cap g_{i, j}^{-1}\Gamma_v g_{i, j} = H(k_v)\cap h_v g_i^{-1} \gamma_v^{-1} \Gamma_v \gamma_v g_i h_v^{-1}=  H(k_v)\cap h_v g_i^{-1}  \Gamma_v  g_i h_v^{-1}= H(k_v)\cap  g_i^{-1} \Gamma_v  g_i $$ for all $v\leq \infty_k$. This implies that 
\begin{equation} \label{inner-form-equality} \mu_{\rmH_{y_{i, j}}^{\an}}{(\rmH_{y_{i, j}}^{\an}/\rmH_{y_{i, j}}^{\an}\cap \Gamma})=\mu_{\rmH_{x_{i}}^{\an}}{(\rmH_{x_{i}}^{\an}/\rmH_{x_{i}}^{\an}\cap \Gamma}) \end{equation}
for all $1\leq j\leq l_{x_i}$ by (\ref{inner-tamagawa-y}) and (\ref{inner-tamagawa-x}).

Since $$ B_T= B_{i, j}(T) \ \ \text{ for $1\leq j\leq h_{x_i}$ and $1\leq i\leq l$ } $$ by Lemma \ref{tran}, one concludes that 
\begin{equation} \label{count}  N({\mathcal X}, T) \sim \sum_{i=1}^l   \frac{h_{x_i} \cdot \mu_{\rmH_{x_i}^{\an}}{(\rmH_{x_i}^{\an}/\rmH_{x_i}^{\an}\cap \Gamma})}
     {\lambda_{\infty_k}{(G(k\otimes_{\Bbb Q} \bbR)/\Gamma)}}\cdot 
     \log{T}  \int_{B_{T}} \nu_{\infty_k}   \end{equation} as $T\to \infty$ by (\ref{inner-form-equality}).  
 Since $h_{x_i}$ is equal to 
$$ \begin{aligned} & [H_{x_i}({\bf A}_k): H_{x_i}(k) \prod_{v \leq \infty_k} ( H_{x_i}(k_v)\cap\Gamma_v )]) =  [H({\bf A}_k): H(k) \prod_{v \leq \infty_k} ( H(k_v)\cap g_i^{-1}\Gamma_v g_i)]) \\
= \  & h_k [H(k) \prod_{v \leq \infty_k} N_v : H(k) \prod_{v \leq \infty_k} ( H(k_v)\cap g_i^{-1}\Gamma_v g_i)] \\
= \ & h_k [(H(k) \cap  \prod_{v< \infty_k} N_v) : (H(k)\cap \prod_{v< \infty_k} ( H(k_v)\cap g_i^{-1}\Gamma_v g_i))]^{-1} \prod_{v < \infty_k} [N_v : ( N_v\cap g_i^{-1}\Gamma_v g_i)]
\end{aligned}$$
 for $1\leq i\leq l$ by (\ref{tama}) and (\ref{transv-x}), one obtains that 
\begin{equation}\label{key} h_{x_i} \cdot \mu_{\rmH_{x_i}^{\an}}{(\rmH_{x_i}^{\an}/\rmH_{x_i}^{\an}\cap \Gamma}) = \frac{2^r (2\pi)^s}{w_k} R_k h_k \prod_{v < \infty_k} [N_v : ( N_v\cap g_i^{-1}\Gamma_v g_i)] \end{equation}
for $1\leq i\leq l$ by (\ref{inner-tamagawa-x}). 
Since the Tamagawa number of $G$ is 1 by \cite[Chapter IV; Thm.4.4.1]{W}, one obtains
$$ 1= \lambda_G( G(k) \backslash G({\bf A}_{k})) = \lambda_G( G(k) \backslash (G(k) \cdot \prod_{v \leq \infty_k} \Gamma_v)) $$ by strong approximation for $G$ (see \cite[Theorem 7.12]{PR94}), where $\lambda_G=\prod_{v \leq \infty_k}\lambda_v$. This implies that 
$$  \lambda_{\infty_k} (G(k\otimes_{\Bbb Q} \Bbb R)/\Gamma)^{-1}= \prod_{v< \infty_k} \int_{\Gamma_v} \lambda_v  $$ by reduction theory (see \cite[Chapter 5, \S 5.2]{PR94}). 
Partitioning the set $\{x_i\}_{i=1}^l$ into the following disjoint union 
$$ \{x_i\}_{i=1}^l = \prod_{v\in (S\setminus \infty_k)} \{ x_{v, j}: \ 1\leq j\leq l_v \} \ \ \ \text{such that} \ \ \  \prod_{v\in (S\setminus \infty_k)} l_v= l \ \ \ \text{and} \ \ \ \mathcal X(\frak  o_{k_v})= \bigcup_{j=1}^{l_v}\Gamma_v x_{v, j} $$ for $v\in (S\setminus \infty_k)$ by Lemma \ref{loc} and writing $g_{v, j}$ for $g_i$ in (\ref{transv-x}) when we use $x_{v, j}$ for $x_i$, one obtains  
\begin{equation} \label{local-index}  \begin{aligned} & \int_{\mathcal X(\frak o_{k_v})} \nu_v= \sum_{j=1}^{l_v} ( \int_{\Gamma_v} \lambda_v)(\int_{H(k_v) \cap g_{v, j}^{-1}\Gamma_v g_{v, j}}\mu_v)^{-1} \\
  = \ & (\sum_{j=1}^{l_v} [N_v : N_v\cap g_{v, j}^{-1}\Gamma_v g_{v, j}] ) ( \int_{\Gamma_v} \lambda_v)(\int_{N_v}\mu_v)^{-1} 
  \end{aligned} \end{equation} for all primes $v\in (S\setminus \infty_k)$. For any primes $v\not\in S$, one simply has 
\begin{equation} \label{general}  \int_{\mathcal X(\frak o_{k_v})} \nu_v \cdot \int_{N_v}\mu_v=  \int_{\Gamma_v} \lambda_v \end{equation}  by (\ref{int-mod}).  Combining (\ref{count}), (\ref{key}), (\ref{local-index}) and (\ref{general}), one concludes that  
\begin{equation*}
    \begin{aligned}
& N({\mathcal X}, T) \sim  \frac{2^r (2\pi)^s}{w_k} R_k h_k \cdot  (\sum_{i=1}^l   \prod_{v < \infty_k} [N_v : ( N_v\cap g_i^{-1}\Gamma_v g_i)] \int_{\Gamma_v} \lambda_v ) \cdot  (\log{T}  \int_{B_{T}} \nu_{\infty_k})  \\
& = \frac{2^r (2\pi)^s}{w_k} R_k h_k \cdot ((\prod_{v\not\in S} \int_{\Gamma_v} \lambda_v)( \sum_{i=1}^l   \prod_{v \in (S \setminus \infty_k)} [N_v : ( N_v\cap g_i^{-1}\Gamma_v g_i)] \int_{\Gamma_v} \lambda_v )) \cdot  (\log{T}  \int_{B_{T}} \nu_{\infty_k}) \\
& = \frac{2^r (2\pi)^s}{w_k} R_k h_k \cdot ((\prod_{v\not\in S} \int_{\Gamma_v} \lambda_v) ( \prod_{v \in (S \setminus \infty_k)} \sum_{j=1}^{l_v}  [N_v : ( N_v\cap g_{v, j}^{-1}\Gamma_v g_{v, j})] \int_{\Gamma_v} \lambda_v )) \cdot  (\log{T}  \int_{B_{T}} \nu_{\infty_k}) \\
& = \frac{2^r (2\pi)^s}{w_k} R_k h_k \cdot ((\prod_{v\not\in S} \int_{\Gamma_v} \lambda_v) ( \prod_{v \in (S \setminus \infty_k)}  \int_{\mathcal X(\frak o_{k_v})} \nu_v \cdot \int_{N_v}\mu_v)) \cdot  (\log{T}  \int_{B_{T}} \nu_{\infty_k}) \\
& = \frac{2^r (2\pi)^s}{w_k} R_k h_k \cdot (\prod_{v< \infty_k}  \int_{\mathcal X(\frak o_{k_v})} \nu_v \cdot \int_{N_v}\mu_v) \cdot  (\log{T}  \int_{B_{T}} \nu_{\infty_k})
\end{aligned}
 \end{equation*}
  as $T\to \infty$.  Since 
  $$\int_{N_v}\mu_v = (1-q_v^{-1})[\frak o_{k_v} : D_{k_v}]^{-\frac{1}{2}} \ \ \ \text{and} \ \ \ \prod_{v<\infty_k} [\frak o_{k_v} : D_{k_v}] = |d_k| $$
by  \cite[Lemma 2.3.3]{Tate} and \cite[Chapter III, (2.9) Theorem and (2.11) Corollary]{N} where $D_{k_v}$ is the absolute different of $k_v$ for all $v<\infty_k$, one obtains the desired result.
\end{proof} 

If $k=\Bbb Q$, then $s=0$, $w_k=2$ and $r=d_k=R_k=h_k=1$ in Theorem.\ref{main}. One obtains Theorem \ref{intro-m} by \cite[(1.8.0) and Lemma 1.8.1]{BR95}.

\section{An example}

As an application of Theorem \ref{main}, we will give a precise asymptotic formula for counting the number of integral points of the following equation 
\begin{equation} \label{exam-equ} \mathcal X: \ \  \ x^2+y^2 - \delta z^2=1 \end{equation} with $\delta\in \Bbb Z$ and $\delta >0$, which has been studied in \cite{DRS}. We keep the same notations as those in the previous sections.

\begin{lem}\label{diff-form} Let $X= \mathcal X \times_{\Bbb Z} \Bbb Q$. Then  
$$ \omega_X = \frac{d x\wedge d y}{-2 \delta z} =  \frac{ dz\wedge dy}{-2x} =  \frac{dx \wedge dz}{-2y} $$
is a gauge form on $X$ such that the induced Tamagawa measure $\nu_p$ on $X(\Bbb Q_p)$ satisfies 
\begin{equation}\label{measure-density}
\int_{\mathcal X(\Bbb Z_p)} \nu_p= \lim_{k\to \infty} \frac{\#\{ (\beta_1, \beta_2, \beta_3)\in (\Bbb Z/(p^k))^3: \ f(\beta_1, \beta_2, \beta_3) \equiv 1 \mod p^k\}}{p^{2k}}\end{equation} for all primes $p$. 
\end{lem}
\begin{proof}  The equality of $\omega_X$ follows from 
$$x \cdot dx + y\cdot dy -\delta z\cdot dz =0  $$ by differentiating the given equation (\ref{exam-equ}). 

Let $SO(f)$ be the special orthogonal group defined by $f=x^2+y^2-\delta z^2$. Since $X$ is also a homogeneous space of $SO(f)$ by \cite[\S 5.6]{CTX}, it is enough to show that $\omega_X$ is $SO(f)$-invariant by \cite[Theorem 2.4.1]{W}. Let 
$$ (x', y', z')=(x, y, z) \cdot A \ \ \ \text{with} \ \ \ A=  \left[\begin{array} [c]{llll}
a_{11} \ \  a_{12}  \ \ a_{13}  \\ a_{21} \ \ a_{22} \ \ a_{23} \\
a_{31} \ \ a_{32}  \ \ a_{33}
\end{array}
\right]\in SO(f) . $$
Since 
$$ A \cdot \left[\begin{array} [c]{llll}
1 \ \  & 0  \ \  & \ \ 0  \\  0 \ \ & 1 \ \  & \ \ 0 \\
0 \ \ & 0  \ \  & -\delta
\end{array}
\right]= \left[\begin{array} [c]{llll}
1 \ \  & 0  \ \  & \ \ 0  \\  0 \ \ & 1 \ \  & \ \ 0 \\
0 \ \ & 0  \ \  & -\delta
\end{array}
\right] \cdot (A')^{-1} ,$$ one obtains 
\begin{equation} \label{matrix-equ} \left[\begin{array} [c]{llll}
a_{11} \ \  & a_{12}  \ \  & -\delta \cdot a_{13}  \\ a_{21} \ \ & a_{22} \ \ & -\delta \cdot a_{23} \\
a_{31} \ \ & a_{32}  \ \ &-\delta \cdot a_{33}
\end{array}
\right] = \left[\begin{array} [c]{llll}
\ \ A_{11} \ \  & \ \ A_{12}  \ \ & \ \ A_{13}  \\  \ \ A_{21} \ \  & \ \ A_{22} \ \ & \ \ A_{23} \\
-\delta\cdot A_{31} \ \ & -\delta\cdot A_{32}  \ \ & -\delta\cdot A_{33}
\end{array}
\right] \end{equation} where $A_{ij}$ is the cofactor of $a_{ij}$ in $A$ for $1\leq i, j \leq 3$. Therefore
$$ dx'\wedge dy'=A_{33} \cdot dx\wedge dy +A_{13}\cdot  dy\wedge dz + A_{23} \cdot dz\wedge dx $$
$$ = (A_{33} - A_{13}  \frac{x}{\delta z} -A_{23} \frac{y}{\delta z})\cdot  dx\wedge dy = \frac{z'}{z} \cdot dx\wedge dy $$ by the equality of $\omega_X$ and (\ref{matrix-equ}). 

In order to show (\ref{measure-density}), one considers a morphism 
$$  \psi: \Spec(k[x, y, z]) \longrightarrow \Spec(k[t]) ;  \ (x, y, z) \mapsto f(x, y, z)-1 $$ induced by $f$. Since $\omega_X$ satisfies 
$$ \omega_X \wedge \psi^*(dt)=(\frac{d x\wedge d y}{-2 \delta z}) \wedge (2x \cdot dx + 2y\cdot dy -2\delta z\cdot dz)=dx\wedge dy\wedge dz,  $$ one concludes that (\ref{measure-density}) holds by \cite[Lemma 1.8.1]{BR95}. 
\end{proof}

The local densities of $f$ can be computed as follows.  

\begin{lem} \label{local-density} If $f(x, y, z)=x^2+y^2 - \delta z^2$ with $\delta\in \Bbb Z$ and $\delta >0$ and 
$$ \alpha_p (f, 1) =  \lim_{k\to \infty} \frac{\#\{ (\beta_1, \beta_2, \beta_3)\in (\Bbb Z/(p^k))^3: \ f(\beta_1, \beta_2, \beta_3) \equiv 1 \mod p^k\}}{p^{2k}} $$ for all primes $p$, then 
$$ \alpha_p (f, 1)= \begin{cases}  1+ (\frac{\delta}{p}) p^{-1}  \ \ \ & \text{$p$ is odd with $(p,\delta)=1$} \\
1-(\frac{-1}{p }) p^{-1} \ \ \ & \text{$p$ is odd with $p| \delta$} \\
1 \ \ \ & \text{$p=2$ with $4 \mid \delta$ or $\delta\equiv 1 \mod 8$} \\
3\cdot 2^{-2}  \ \ \ & \text{$p=2$ with $\delta\equiv 3, 7 \mod 8$}\\
2^{-1} \ \ \ & \text{$p=2$ with $2| \delta$ but $4 \nmid \delta$ or $\delta\equiv 5 \mod 8$}
\end{cases} $$
where $(\frac{*}{p})$ are the Legendre symbols.
\end{lem}
\begin{proof} When $p$ is odd and $(p,\delta)=1$, then  $$\alpha_p (f, 1)=  1+ (\frac{\delta}{p}) p^{-1}$$ by \cite[Hilfssatz 12]{S}. 

\medskip
When $p$ is odd with $p| \delta$ and $(\frac{-1}{p})=1$, then $f$ is equivalent to $xy-\delta z^2$ over $\Bbb Z_p$. Since 
$$\#\{ (\beta_1, \beta_2, \beta_3)\in (\Bbb Z/(p^k))^3: \ \beta_1\beta_2-\delta\beta_3^2 \equiv 1 \mod p^k\}= (p^k-p^{k-1}) p^k $$ for sufficiently large $k$, one obtains that $\alpha_p (f, 1)=  1-p^{-1}$. 

When $p$ is odd with $p| \delta$ and $(\frac{-1}{p})=-1$, then $$\# \{ (\bar{\xi}, \bar{\eta})\in (\Bbb Z/p\Bbb Z)^2: \ \ \ \bar{\xi}^2+\bar{\eta}^2 =\bar{1} \mod p\} =p+1 $$ and one of $\xi$ and $\eta$ is not divisible by $p$. For such a pair $(\bar{\xi}, \bar{\eta})$, one can make a substitution $$\begin{cases} x=\xi+p x_1 \\ y=\eta+p y_1 \end{cases}$$ and obtain the equation 
\begin{equation} \label{exam-deform-equ} 2\xi x_1+ p x_1^2 + 2\eta y_1 + p y_1^2 -(\delta p^{-1}) z^2 = (1-\xi^2-\eta^2)p^{-1} \end{equation} over $\Bbb Z$ by (\ref{exam-equ}). Without loss of generality, we assume that $\xi$ is not divisible by $p$. By Hensel's Lemma, the equation (\ref{exam-deform-equ}) has a unique solution for $x_1$ over $\Bbb Z_p$ for $y_1$ and $z$ taking any values in $\Bbb Z$. This implies that
$$\#\{ (\beta_1, \beta_2, \beta_3)\in (\Bbb Z/(p^k))^3: \ f(\beta_1, \beta_2, \beta_3) \equiv 1 \mod p^k\}= (p+1) p^{k-1} p^k  $$ for sufficiently large $k$ and $\alpha_p (f, 1)=  1+p^{-1}$. 

\medskip

When $p=2$ and $(\delta, 2)=1$, then
\begin{equation}\label{canonical-form} f \cong \begin{cases} 2xy +\delta z^2 \ \ \ & \text{if $(\frac{-1, -\delta}{\Bbb Q_2})=-1$}  \\
 2(x^2+xy+y^2) -3\delta z^2  \ \ \ & \text{if $(\frac{-1, -\delta}{\Bbb Q_2})=1$} \end{cases} \end{equation} over $\Bbb Z_2$ by \cite[93:18. Example (iv) and 58:6]{OM}. 

$\bullet$ Case $\delta\equiv 1\mod 8$. Then $f\cong 2xy+z^2$ by (\ref{canonical-form}). Since an odd integer $m$ is a square in $\Bbb Z_2$ if and only if $m\equiv 1\mod 8$ by \cite[63:1. Local Square Theorem]{OM}, one concludes that 
$$\#\{ (\beta_1, \beta_2, \beta_3)\in (\Bbb Z/(2^k))^3: \ 2\beta_1 \beta_2+\beta_3^2 \equiv 1 \mod 2^k\}=2(2^{2k} - 2^{2(k-1)}-2\cdot 2^{k-2} \cdot 2^{k-1}) $$ for sufficiently large $k$ and $\alpha_2 (f, 1)= 1$.

$\bullet$ Case $\delta \equiv 3 \mod 8$.  Then $f\cong 2(x^2+xy+y^2)-z^2$ by (\ref{canonical-form}). Since $z$ and at least one of $x$ and $y$ take odd integers for the equation $2(x^2+xy+y^2)-z^2=1$, one can write $z^2=1+8t$. 

Suppose both $x$ and $y$ are odd. Write $x=1+2x_1$ and $y=1+2y_1$. Then
\begin{equation} \label{2}  \frac{1}{2}(3x_1+3y_1+1)+x_1^2+x_1y_1+y_1^2= t \end{equation}  and $x_1$ and $y_1$ have a different parity. Moreover, for any even (resp. odd) $x_1$, there is a unique odd (resp. even) $y_1\in \Bbb Z_2$ satisfying (\ref{2}) by Hensel's Lemma.  

Suppose $x$ is odd and $y$ is even. Write $x=1+2x_1$ and $y=2y_1$. Then
\begin{equation} \label{3} 2x_1(x_1+1) + y_1(1+2x_1) + 2y_1^2= 2t . \end{equation}
For any $x_1\in \Bbb Z_2$, there is a unique $y_1\in \Bbb Z_2$ satisfying (\ref{3}) by Hensel's Lemma. 

Summarizing the above arguments with symmetry of $x$ and $y$, one concludes that 
$$\#\{ (\beta_1, \beta_2, \beta_3)\in (\Bbb Z/(2^k))^3: \ 2(\beta_1^2+\beta_1 \beta_2+\beta_2^2)-\beta_3^2 \equiv 1 \mod 2^k\}=2^{k-1}(2\cdot 2^{k-2}+ 2\cdot 2^{k-1})$$
for sufficiently large $k$ and $\alpha_2 (f, 1)=3\cdot 2^{-2}$.

$\bullet$ Case $\delta \equiv 5 \mod 8$. Then $f\cong 2xy+5 z^2$ by (\ref{canonical-form}). By \cite[63:1 Local Square Theorem]{OM}, one concludes  
$$\#\{ (\beta_1, \beta_2, \beta_3)\in (\Bbb Z/(2^k))^3: \ 2\beta_1 \beta_2+5\beta_3^2 \equiv 1 \mod 2^k\}=2 \cdot (2\cdot 2^{k-1}\cdot 2^{k-2})$$
for sufficiently large $k$ and $\alpha_2 (f, 1)=  2^{-1}$. 

$\bullet$ Case $\delta \equiv 7 \mod 8$. Then $f\cong 2(x^2+xy+y^2)-5z^2$ by (\ref{canonical-form}). By \cite[63:1 Local Square Theorem]{OM}, the equation $2(x^2+xy+y^2)-5z^2=1$ is solvable over $\Bbb Z_2$ if and only if  
\begin{equation} \label{4} (x^2+xy+y^2-1)\in 2 \Bbb Z_2^\times . \end{equation}

Suppose both $x$ and $y$ are odd. Write $x=1+2x_1$ and $y=1+2y_1$. Then $(x_1+y_1)\in 2\Bbb Z_2$ by (\ref{4}). Suppose $x$ is even and write $x=2x_1$. Then both $x_1$ and  $y$ have to be odd by (\ref{4}). The similar result is true for even $y$. Therefore 
$$ \begin{aligned} &  \#\{ (\beta_1, \beta_2, \beta_3)\in (\Bbb Z/(2^k))^3: \ 2(\beta_1^2+\beta_1 \beta_2+\beta_2^2)-5\beta_3^2 \equiv 1 \mod 2^k\} \\
& =  2\cdot (2^{2(k-1)}-2\cdot 2^{2(k-2)} + 2\cdot 2^{k-2}\cdot 2^{k-1})= 3 \cdot 2^{2(k-1)} \end{aligned}$$
for sufficiently large $k$ and $\alpha_2 (f, 1)=3\cdot 2^{-2}$. 

\medskip

When $p=2$ and $2| \delta$ but $4 \nmid \delta$, then $f$ is equivalent to $$ (\Bbb Z_2 e_1 + \Bbb Z_2 e_2) \perp \Bbb Z_2 e_3 \ \ \ \text{with}  \ \ Q(e_1)=1,  B(e_1,e_2)=1, Q(e_2)=2 \ \ \text{and} \ Q(e_3)= \delta $$ by \cite[93:17. Example]{OM}. Since  $2| \delta$ but $4 \nmid \delta$, there is $\kappa\in \Bbb Z_2$ such that $$Q(e_2+e_3)=2+\delta= 4\kappa . $$ Since $\Bbb Z_2e_1 + \Bbb Z_2(e_2+e_3)$ splits $(\Bbb Z_2 e_1 + \Bbb Z_2 e_2) \perp \Bbb Z_2 e_3$ by \cite[82:15]{OM}, one obtains that 
$$ (\Bbb Z_2 e_1 + \Bbb Z_2 e_2) \perp \Bbb Z_2 e_3 = (\Bbb Z_2e_1 + \Bbb Z_2(e_2+e_3)) \perp \Bbb Z_2 e_3'  $$ with $Q(e_3')= (1-4\kappa)^{-1}\delta$ by comparing the discriminants of lattices. 
One concludes that 
    $$f\cong x^2+ 2xy +4\kappa y^2+ (1-4\kappa)^{-1}\delta z^2$$ over $\Bbb Z_2$.  Therefore 
$$ \#\{ (\beta_1, \beta_2, \beta_3)\in (\Bbb Z/(2^k))^3: \ \beta_1^2+2\beta_1 \beta_2+4\kappa \beta_2^2+(1-4\kappa)^{-1} \delta\beta_3^2 \equiv 1 \mod 2^k\} =2^{k-1} \cdot 2^k$$
for sufficiently large $k$ and $\alpha_2 (f, 1)= 2^{-1}$.

\medskip

When $p=2$ and $4 \mid \delta$, then $f\cong x^2+2xy+2y^2-\delta z^2$ over $\Bbb Z_2$ by \cite[93:17. Example]{OM}. Therefore 
$$ \#\{ (\beta_1, \beta_2, \beta_3)\in (\Bbb Z/(2^k))^3: \ \beta_1^2+2\beta_1 \beta_2+2 \beta_2^2- \delta\beta_3^2 \equiv 1 \mod 2^k\} =2\cdot 2^{k-1} \cdot 2^k$$
for sufficiently large $k$ and $\alpha_2 (f, 1)= 1$.
\end{proof}

The singular integral of $\mathcal X$ can be computed as follows.

\begin{lem} \label{infinite-part} If $\delta>0$ and 
$$ B_T= \{(x, y, z)\in \Bbb R^3:  \ x^2+y^2-\delta z^2 =1\ \text{with} \ \sqrt{x^2+y^2+z^2} \leq T \}  $$ for $T\geq 1$, then
$$ \int_{B_T} \frac{1}{2\delta z} dx dy = 2 \pi\sqrt{\frac{T^2-1}{1+\delta}} .$$
\end{lem}
\begin{proof} Since $z^2=\delta^{-1}(x^2+y^2-1)\geq 0$, one obtains that 
$$ \int_{B_T} \frac{1}{2\delta z} dx dy =  \int_{D_T} \frac{1}{\delta z} dx dy $$
where $D_T$ can be described as 
$$ \{ (x, y)\in \Bbb R^2: \ \sqrt{\frac{\delta}{1+\delta}} \leq \sqrt{x^2+y^2-(\delta+1)^{-1}} \leq \sqrt{\frac{\delta}{1+\delta}} \cdot T \} . $$
By using the polar coordinates $x=\rho \cos\theta$ and $y=\rho \sin \theta$, one obtains 
$$ \int_{B_T} \frac{1}{2\delta z} dx dy =\frac{2\pi}{\sqrt{\delta}} \int_{1}^{\sqrt{\frac{\delta T^2+1}{1+\delta}}} \frac{\rho}{\sqrt{\rho^2-1}} d\rho= 2\pi\sqrt{\frac{T^2-1}{1+\delta}} $$
as desired.
\end{proof}

Write $\delta= \delta_0\delta_1^2$ where both $\delta_0$ and $\delta_1$ are positive integers and $\delta_0$ is square free. Let
$$ S= \{ p \ \text{primes}: \ p\mid \delta_1 \ \text{and} \ (p, 2\delta_0)=1 \} . $$

\begin{exa} \label{example} 
Suppose that $f(x, y, z)=x^2+y^2 - \delta z^2$ and  
$$ N(f, 1, T)= \# \{(\alpha, \beta, \gamma)\in \Bbb Z^3:  \ f(\alpha, \beta, \gamma) =1\ \text{with} \ \sqrt{\alpha^2+\beta^2+\gamma^2} \leq T \}  $$ for $T>0$.

If $\delta_0>1$, then 
$$ N(f, 1, T)\sim  \frac{c}{\pi \sqrt{1+\delta}} \cdot  L(1, \chi)\cdot \frac{\prod_{p\in S} (1-(\frac{-1}{p})p^{-1})(1+(\frac{\delta_0}{p})p^{-1})^{-1}} {\prod_{p\mid \delta_0, \ p\neq 2} (1+(\frac{-1}{p})p^{-1})} \cdot T $$
as $T\to \infty$ where  
$$ c= \begin{cases}  8 \ \ \ & \text{$\delta_0 \equiv 1 \mod 8$}  \\
12 \ \ \ & \text{$\delta_0 \equiv 5 \mod 8$ and $2\nmid \delta$} \\
24 \ \ \ & \text{$\delta_0 \equiv 5 \mod 8$ and $2 \mid \delta$} \\
12 \ \ \ & \text{$\delta_0\equiv 3, 7 \mod 8$ and $2 \nmid \delta$} \\
16 \ \ \ & \text{$\delta_0\equiv 3, 7 \mod 8$ and $2\mid \delta$} \\
8 \ \ \ & \text{$2\mid \delta_0$ and $2\nmid \delta_1$} \\
16 \ \ \ & \text{$2\mid \delta_0$ and $2\mid \delta_1$}
\end{cases}$$ 
and $L(\chi, s)$ is the Dirichlet $L$-function associated to the non-trivial quadratic character $\chi$ of $\Gal(\Bbb Q(\sqrt{\delta_0})/\Bbb Q)$.

If $\delta_0=1$, then 
$$ N(f, 1, T)\sim  \frac{8}{\pi \sqrt{1+\delta}} \cdot ( \prod_{p\mid \delta,\  p\neq 2} \frac{p-(\frac{-1}{p})}{p+1} )\cdot  T \log T $$ as $T\to \infty$. 
\end{exa}

\begin{proof}
By \cite[102:7 and 104:5 Theorem]{OM}, the number $h(f)$ of classes in $gen(f)$ is given by the formula
$$ h(f)=[\Bbb G_m({\bf A}_{\Bbb Q}) : \theta(SO(f_{\Bbb Q}))\cdot (\theta (SO(f_{\Bbb R})\times \prod_{p \ \text{primes}} SO(f_{\Bbb Z_p}))) ]$$
where $\theta$ is the spinor norm map.  Since 
$$ \theta(SO(f_{\Bbb Q}))= \Bbb Q^\times,  \ \ \ \theta (SO(f_{\Bbb R})) =\Bbb R^\times  \ \ \ \text{and} \ \ \  \theta(SO(f_{\Bbb Z_p}))\supseteq \Bbb Z_p^\times $$
for odd primes $p$ by \cite[101:8, 55:2a and 92:5]{OM}, then $-1\in \theta(SO(f_{\Bbb Q}))$ implies that 
 $$(i_p)_{p\leq \infty_{\Bbb Q}}\in \theta(SO(f_{\Bbb Q}))\cdot (\theta (SO(f_{\Bbb R})\times \prod_{p \ \text{primes}} SO(f_{\Bbb Z_p})))$$ where 
$ i_p =  -1$ for $p=2$ and $i_p=1$ otherwise. Since $5$ is represented by $f$ over $\Bbb Z_2$, one has $5\in \theta(SO(f_{\Bbb Z_2}))$. Therefore 
$$\theta(SO(f_{\Bbb Q}))\cdot (\theta (SO(f_{\Bbb R})\times \prod_{p \ \text{primes}} SO(f_{\Bbb Z_p}))) \supseteq \Bbb Q^\times \cdot  (\Bbb R^\times \times (\prod_{p \ \text{primes}} \Bbb Z_p^\times ))$$ and $h(f)=1$. 

\medskip

If $\delta_0>1$,  then  
$$ N(f, 1, T) \sim (\prod_{p \ \text{primes}} \alpha_p(f, 1)) \cdot   \int_{B_T} \frac{1}{2\delta z} dx dy $$ as $T\to \infty$ by \cite[\S 5]{CX} and Lemma \ref{diff-form}, where $\alpha_p(f, 1)$ and $B_T$ are defined in Lemma \ref{local-density} and  Lemma \ref{infinite-part} respectively. 

$\bullet$ Case $\delta_0\equiv 1\mod 8$.  Then $2$ splits completely in $\Bbb Q(\sqrt{\delta_0})$ and
$$ L(\chi, 1)= 2 \cdot \prod_{p\nmid 2\delta_0} (1-(\frac{\delta_0}{p})p^{-1})^{-1} .$$
Therefore 
$$ \begin{aligned}&  N(f, 1, T) \sim  \prod_{p\mid \delta,\  p\neq 2} (1-(\frac{-1}{p})p^{-1}) \cdot \prod_{p\nmid 2\delta} (1+(\frac{\delta_0}{p})p^{-1}) \cdot  \frac{2\pi}{\sqrt{1+\delta}} T \\
= & \frac{2\pi}{\sqrt{1+\delta}} \cdot \frac{\prod_{p\mid \delta,\  p\neq 2} (1-(\frac{-1}{p})p^{-1})}{\prod_{p\in S} (1+(\frac{\delta_0}{p})p^{-1})}\cdot \prod_{p\nmid 2\delta_0} \frac{1-p^{-2}}{1-(\frac{\delta_0}{p})p^{-1}} \cdot T \\
= & \frac{2\pi}{\sqrt{1+\delta}} \cdot \frac{\prod_{p\mid \delta,\  p\neq  2} (1-(\frac{-1}{p})p^{-1})}{\prod_{p\in S} (1+(\frac{\delta_0}{p})p^{-1})}\cdot 2^{-1} L(1, \chi) \cdot \frac{4}{3}\prod_{p\mid \delta_0} (1-p^{-2})^{-1}\cdot \zeta(2)^{-1} \cdot T \\
= & \frac{8}{\pi \sqrt{1+\delta}} \cdot  L(1, \chi)\cdot \frac{\prod_{p\mid \delta,\  p\neq  2} (1-(\frac{-1}{p})p^{-1})}{\prod_{p\in S} (1+(\frac{\delta_0}{p})p^{-1})} \cdot \prod_{p\mid \delta_0} (1-p^{-2})^{-1} \cdot T \\
= & \frac{8}{\pi \sqrt{1+\delta}} \cdot  L(1, \chi)\cdot \frac{\prod_{p\in S} (1-(\frac{-1}{p})p^{-1})(1+(\frac{\delta_0}{p})p^{-1})^{-1}} {\prod_{p\mid \delta_0} (1+(\frac{-1}{p})p^{-1})} \cdot T 
\end{aligned}$$
as $T\to \infty$ by Lemma \ref{local-density} and  Lemma \ref{infinite-part} and $\zeta(2) = \frac{\pi^2}{6}$. 

$\bullet$ Case $\delta_0\equiv 5\mod 8$. Then $2$ is inert in $\Bbb Q(\sqrt{\delta_0})$ and 
$$ L(\chi, 1)= \frac{2}{3} \cdot \prod_{p\nmid 2\delta_0} (1-(\frac{\delta_0}{p})p^{-1})^{-1} .$$
Therefore 
$$ \begin{aligned} & N(f, 1, T) \sim \begin{cases}  \frac{1}{2}\cdot  \prod_{p\mid \delta} (1-(\frac{-1}{p})p^{-1}) \cdot \prod_{p\nmid 2\delta} (1+(\frac{\delta_0}{p})p^{-1}) \cdot  \frac{2\pi}{\sqrt{1+\delta}} T \ \ \ & 2\nmid \delta \\
\prod_{p\mid \delta,\  p \neq  2} (1-(\frac{-1}{p})p^{-1}) \cdot \prod_{p\nmid \delta} (1+(\frac{\delta_0}{p})p^{-1}) \cdot  \frac{2\pi}{\sqrt{1+\delta}} T \ \ \ & 2\mid \delta \end{cases} \\
\end{aligned}$$
as $T\to \infty$ by Lemma \ref{local-density} and  Lemma \ref{infinite-part} and the result follows.  

$\bullet$ Case $\delta_0\equiv 3, 7 \mod 8$. Then $2$ is ramified in $\Bbb Q(\sqrt{\delta_0})$ and
$$ L(\chi, 1)= \prod_{p\nmid 2\delta_0} (1-(\frac{\delta_0}{p})p^{-1})^{-1} .$$
Therefore 
$$ \begin{aligned} & N(f, 1, T) \sim \begin{cases} \frac{3}{4}\cdot  \prod_{p\mid \delta} (1-(\frac{-1}{p})p^{-1}) \cdot \prod_{p\nmid 2\delta} (1+(\frac{\delta_0}{p})p^{-1}) \cdot  \frac{2\pi}{\sqrt{1+\delta}} T  \ \ \ &   2\nmid \delta  \\
\prod_{p\mid \delta,\  p \neq 2} (1-(\frac{-1}{p})p^{-1}) \cdot \prod_{p\nmid \delta} (1+(\frac{\delta_0}{p})p^{-1}) \cdot  \frac{2\pi}{\sqrt{1+\delta}} T \ \ \ & 2\mid \delta \end{cases} 
\end{aligned}$$
as $T\to \infty$ by Lemma \ref{local-density} and  Lemma \ref{infinite-part} and the result follows.

$\bullet$ Case $2\mid \delta_0$. Then $2$ is ramified in $\Bbb Q(\sqrt{\delta_0})$ and
$$ L(\chi, 1)= \prod_{p\nmid \delta_0} (1-(\frac{\delta_0}{p})p^{-1})^{-1} .$$
Therefore 
$$ \begin{aligned} & N(f, 1, T) \sim \begin{cases} \frac{1}{2}\cdot  \prod_{p\mid \delta,\  p\neq 2} (1-(\frac{-1}{p})p^{-1}) \cdot \prod_{p\nmid \delta} (1+(\frac{\delta_0}{p})p^{-1}) \cdot  \frac{2\pi}{\sqrt{1+\delta}} T  \ \ \ &  2\nmid \delta_1 \\
\prod_{p\mid \delta,\  p\neq 2} (1-(\frac{-1}{p})p^{-1}) \cdot \prod_{p\nmid \delta} (1+(\frac{\delta_0}{p})p^{-1}) \cdot  \frac{2\pi}{\sqrt{1+\delta}} T \ \ \ & 2\mid \delta_1 \end{cases} \end{aligned}$$
as $T\to \infty$ by Lemma \ref{local-density} and  Lemma \ref{infinite-part} and the result follows.

\medskip

Otherwise $\delta_0=1$.  By Theorem \ref{main}, Lemma \ref{diff-form}, Lemma \ref{local-density} and  Lemma \ref{infinite-part}, one obtains 
$$ \begin{aligned} & N(f, 1, T) \sim (\prod_{p \ \text{primes}} (1-p^{-1}) \alpha_p(f, 1)) \cdot  \log T \int_{B_T} \frac{1}{2\delta z} dx dy \\
= & \frac{1}{2} \prod_{p\mid \delta,\  p\neq 2} (1-p^{-1}) (1-(\frac{-1}{p})p^{-1}) \cdot \prod_{p\nmid \delta,\  p\neq 2} (1-p^{-2})\cdot  \frac{2\pi}{\sqrt{1+\delta}} T \log T \\
= & \frac{1}{2} \cdot \frac{\prod_{p\mid \delta,\  p\neq 2} (1-p^{-1}) (1-(\frac{-1}{p})p^{-1})}{\prod_{p\mid \delta,\  p\neq 2} (1-p^{-2})} \cdot \frac{4}{3} \zeta(2)^{-1}  \cdot  \frac{2\pi}{\sqrt{1+\delta}} T \log T \\
= &  \frac{8}{\pi \sqrt{1+\delta}} \cdot ( \prod_{p\mid \delta,\  p\neq 2} \frac{1-(\frac{-1}{p})p^{-1}}{1+p^{-1}} )\cdot  T \log T  \end{aligned}$$ as $T\to \infty$. 
\end{proof}


\bigskip



\begin{bibdiv}


\begin{biblist}

\bib {BHC} {article} {
author={ A. Borel},
author={Harish-Chandra}, 
title={Arithmetic subgroups of algebraic groups},
journal= {Ann. of Math},
volume={75},
date= {1962},
Pages={485-535},
}


\bib {BR95} {article} {
   author={M.Borovoi},
   author={Z.Rudnick},
   title={Hardy-Littlewood varieties and semisimple groups},
   journal={Invent. Math.},
   volume={119},
   date={1995},
   number={},
   Pages={37-66},
}

\bib{CX}{article}{
author= {W.K. Chan}
author={F. Xu},
title={On representations of spinor genera}
journal={Compos. Math.}
volume={140}
date={2004}
pages={287-300}}


\bib{CTX} {article} {
    author={J-L.Colliot-Th\'el\`ene},
    author={F.Xu},
    title={Brauer-Manin obstruction for integral points of homogeneous spaces and
         representations by integral quadratic forms},
    journal={Compositio Math.},
    volume={145}
    date={2009},
    Pages={309-363},
}


\bib{DRS} {article} {
    author={W.Duke},
    author={Z.Rudnick},
    author={P.Sanark},
 title={Density of integer points on affine homogeneous varieties},
  journal={Duke Math. J.},
    volume={71},
      date={1993},
    pages={143-179},
    number={}
 }
 
 
 \bib{EM} {article} {
    author={A.Eskin},
    author={C.McMullen},
 title={Mixing, counting, and equidistribution in Lie groups},
  journal={Duke Math. J.},
    volume={71},
      date={1993},
    pages={181-209},
    number={}
 }

 
 \bib{Gi}{book}{
    author={J. Giraud},
     title={Cohomologie non ab\'elienne},
       volume={179},
     publisher={Springer-Verlag},
     place={},
      date={1971},
   journal={ },
    series={Grundlehren},
    number={ },
}


 
 
\bib{EMS} {article} {
    author={A.Eskin},
    author={S.Mozes},
    author={N.Shah},
 title={Unipotent flows and counting lattice points on homogeneous varieties},
  journal={Ann. of Math.},
    volume={143},
      date={1996},
    pages={253-199},
    number={}
 }



 \bib{Mar}{book}{
    author={G. A. Margulis},
     title={Discrete Subgroups of Semi-simple Lie Groups },
       volume={17},
     publisher={Springer-Verlag},
     place={},
      date={1991},
   journal={ },
    series={A series of modern surveys in mathematics},
    number={ },
}


\bib{N}{book}{
    author={ J.Neukirch},
    title={Algebraic Number Fields},
    volume={322},
    publisher={Springer},
    series={Grundlehren},
    date={1999},
}

 
\bib{OS} {article} {
    author={H. Oh},
    author={N. A. Shah},
 title={Limits of translates of divergent geodesics and integral points on one-sheeted hyperboloids},
  journal={Isr. J. Math. },
    volume={199},
      date={2014},
    pages={915-931},
    number={}
 }



\bib{OM}{book}{
author={  O.T. O'Meara },
title= {Introduction to Quadratic Forms},  
volume={117},
publisher={  Springer-Verlag },
series={Grundlehren },
date={ 1973 },
}


\bib{PR94}{book}{
    author={V. P. Platonov},
    author={A. S. Rapinchuk},
     title={Algebraic groups and  number theory},
     publisher={Academic Press},
     place={},
      date={1994},
    volume={ },
    number={ },
}


 

\bib {S}{article}{
author={ C.L.  Siegel},
title={ \"Uber die analytische Theorie
der quadratischen Formen. I},
 journal={ Ann. of Math. }
 volume={36}
 date= {1935}
pages={527-606}
} 

\bib{Tate}{article}{
author={J. T. Tate}, 
title= {Fourier analysis in number fields and Hecke's zeta-functions}
journal={Algebraic Number Theory, edited by Cassels and Fr\"olich}  
volume={}
date={1967}
pages={305-347}
}












\bib{Vo}{book}{
    author={V.E.Voskresenskii},
     title={Algebraic groups and their birational invariants},
       volume={179},
     publisher={AMS},
     place={},
      date={1998},
   journal={ },
    series={Translations of Mathematical Monographs},
    number={ },
}



 \bib {W}{book}{
    author={A. Weil},
     title={Adeles and Algebraic Groups},
     publisher={Birkh\"auser},
     place={Boston},
      journal={ },
      series={},
    volume={},
    date={1982},
   number={ },
     pages={},
 }


\bib{WX}{article}{
author={D. Wei},
author={F. Xu},
title={Counting integral points in certain homogeneous spaces}
journal={J. of Algebra},
volume={448}
date={2016}
pages={350-398}
}

\bib{Zha}{article}{
    AUTHOR = {Zhang, Runlin},
     TITLE = {Limiting distribution of translates of the orbit of a maximal
              {$\Bbb Q$}-torus from identity on {${\rm SL}_N(\Bbb R)/{\rm
              SL}_N(\Bbb Z)$}},
   JOURNAL = {Math. Ann.},
  FJOURNAL = {Mathematische Annalen},
    VOLUME = {375},
      YEAR = {2019},
    NUMBER = {3-4},
     PAGES = {1231--1281},
      ISSN = {0025-5831},
   MRCLASS = {37A17 (37A05 37A45)},
  MRNUMBER = {4023376},
MRREVIEWER = {Thomas Ward},
       DOI = {10.1007/s00208-019-01896-3},
       URL = {https://doi.org/10.1007/s00208-019-01896-3},
}


\end{biblist}
\end{bibdiv}
\end{document}